\documentclass[reqno,11pt]{amsart}
\usepackage{amssymb,amsmath}
\numberwithin{equation}{section}
\newtheorem{theorem}[equation]{Theorem}
\newtheorem{lemma}[equation]{Lemma}
\newtheorem{proposition}[equation]{Proposition}
\newtheorem{corollary}[equation]{Corollary}

\theoremstyle{definition}
\newtheorem{example}[equation]{Example}
\newtheorem*{remark}{Remark}

\title{Linear recurrence sequences and their convolutions via Bell polynomials}
\author{Daniel Birmajer}
\address{Department of Mathematics\\ Nazareth College\\ 4245 East Ave.\\ Rochester, NY 14618}
\author{Juan B. Gil}
\address{Penn State Altoona\\ 3000 Ivyside Park\\ Altoona, PA 16601}
\author{Michael D. Weiner}

\begin{document}
\maketitle

\begin{abstract}
We recast homogeneous linear recurrence sequences with fixed coefficients in terms of partial Bell polynomials, and use their properties to obtain various combinatorial identities and multifold convolution formulas. Our approach relies on a basis of sequences that can be obtained as the INVERT transform of the coefficients of the given recurrence relation. For such a basis sequence $(y_n)$ with generating function $Y(t)$, and for any positive integer $r$, we give a formula for the convolved sequence generated by $Y(t)^r$ and prove that it satisfies an elegant recurrence relation.
\end{abstract}

%%%%%%%%%%%%%%%%%%%%%%%%%%%%%%%%%%%%%%%%%%%%
\section{Introduction}
\label{sec:introduction}

A linear recurrence sequence $(a_n)$ of elements in a commutative ring $\mathcal R$ is a sequence given by a homogeneous linear recurrence relation
\begin{equation}\label{eq:recRelation}
 a_n=c_1 a_{n-1}+c_2a_{n-2}+\cdots+c_da_{n-d} \;\text{ for } n\ge d,
\end{equation}
with fixed coefficients $c_1,\dots,c_d\in\mathcal R$, together with initial values $a_0, a_1,\dots, a_{d-1}\in\mathcal R$. The generating function of such a sequence is a rational function of the form $f(t)=\frac{p(t)}{q(t)}$ with
\begin{equation*}
 q(t)=1-c_1t-c_2t^2-\cdots-c_dt^d,
\end{equation*}
and a polynomial $p(t)$ of degree at most $d-1$ that depends on the initial values. In other words, the function $f(t)$ is a linear combination of the rational functions 
\begin{equation*}
 \frac{1}{q(t)},\;  \frac{t}{q(t)},\;  \frac{t^2}{q(t)},\, \dots,\,  \frac{t^{d-1}}{q(t)}, 
\end{equation*}
and their respective associated sequences form a basis for the space of linear recurrence sequences with coefficients $c_1,\dots,c_d$. 

On the other hand, if $(c_n)$ is a sequence and $Q(t)$ is the formal power series 
\[ Q(t) = 1-\sum_{n=1}^\infty c_n t^n, \]
then its reciprocal $Y(t)=\frac{1}{Q(t)}$ can be written as $Y(t)=1+\sum_{n=1}^\infty y_n t^n$ with
\begin{equation}\label{eq:canonicalSeq}
 y_n = \sum_{k=0}^n \frac{k!}{n!} B_{n,k}(1!c_1, 2!c_2, 3!c_3, \dots),
\end{equation}
where $B_{n,k}=B_{n,k}(x_1,x_2,\dots)$ denotes the $(n,k)$-th partial Bell polynomial in the variables $x_1,x_2,\dots,x_{n-k+1}$. This is a direct consequence of Fa{\`a} di Bruno's formula (cf. Theorem~B in \cite[Section~3.5]{Comtet}), and the sequence $(y_n)$ is precisely the INVERT\footnote{%
See \cite{BernsteinSloane}. This is the operator $A$ discussed in \cite{Cameron}.} 
transform of $(c_n)$. 

Observe that if $c_n=0$ for all $n>d$, then $1/q(t)=1/Q(t)$. Consequently, any linear recurrence sequence with fixed coefficients can be expressed in terms of partial Bell polynomials in the coefficients of the recurrence. An explicit formula is given in Section~\ref{sec:LinearRecurrence}, see formula \eqref{eq:BellRepresentation}, together with a few illustrating examples. As a particular application, we provide an alternative derivation of the Girard-Waring formulas for the power sum symmetric functions.

The benefit of the representation \eqref{eq:BellRepresentation} is that the partial Bell polynomials absorb the coefficients of the recurrence and facilitate the derivation of universal identities. This is particularly convenient when working with multifold self-convolutions. In Section~\ref{sec:Convolutions}, we recall a convolution formula given by the authors in \cite{BGW14a} and discuss it in the context of multifold convolutions of linear recurrence sequences. For this type of convolved sequences we give a universal recurrence formula (of the same depth as the original sequence), which is obtained by using properties of the partial Bell polynomials. We conclude the paper with a few examples that illustrate some applications of our main result.

%%%%%%%%%%%%%%%%%%%%%%%%%%%%%%%%%%%%%%%%%%%%
\section{Linear Recurrence Sequences}
\label{sec:LinearRecurrence}

The representation of linear recurrence sequences in terms of partial Bell polynomials, as discussed in the introduction, can be summarized as follows:

\begin{proposition}\label{prop:BellRepresenation}
Let $(a_n)$ be a linear recurrence sequence satisfying
\begin{equation*}
 a_n=c_1 a_{n-1}+c_2a_{n-2}+\cdots+c_da_{n-d} \;\text{ for } n\ge d\ge 1,
\end{equation*}
with initial values $a_0, a_1,\dots, a_{d-1}$. Let $(y_n)$ be defined as in \eqref{eq:canonicalSeq}, and let
\begin{equation*}
\lambda_0=a_0, \quad \lambda_n=a_n-\sum_{j=1}^n c_j a_{n-j} \;\text{ for } n=1,\dots,d-1.
\end{equation*}
Then $a_n= \lambda_0 y_n + \lambda_1 y_{n-1}+\cdots+\lambda_{d-1} y_{n-d+1}$, so
\begin{equation}\label{eq:BellRepresentation}
 a_n = \sum_{k=0}^{d-1} \lambda_k \sum_{j=0}^{n-k} \frac{j!}{(n-k)!} B_{n-k,j}(1!c_1, 2!c_2, \dots)\;\text{ for } n\ge 1. 
\end{equation}
\end{proposition}

\begin{proof}
If $S$ denotes the right-shift operator $S(a_1,a_2,\dots)=(0,a_1,a_2,\dots)$, then the sequences $(y_n), S(y_n), S^2(y_n),\dots, S^{d-1}(y_n)$, clearly form a basis for the space of all linear recurrence sequences with coefficients $c_1,\dots,c_d$. Thus there are constants $\lambda_0, \lambda_1,\dots, \lambda_{d-1}$ such that $a_n= \lambda_0 y_n + \lambda_1 y_{n-1}+\cdots+\lambda_{d-1} y_{n-d+1}$, with the convention that $y_k=0$ if $k<0$. To find the $\lambda_k$'s, we just need to look at the initial values and solve the equation
\begin{equation*}
\begin{pmatrix}
 1 & 0 &\cdots&\cdots&\cdots& 0 \\
 y_1 & 1 & 0 &\cdots&\cdots& 0 \\
 y_2 & y_1 & 1 & 0 & \cdots & 0 \\
 \vdots & \vdots &&\ddots&& \vdots \\
 \vdots & \vdots &&&\ddots& \vdots \\
 y_{d-1} & y_{d-2} & \cdots & \cdots & y_1 & 1 
 \end{pmatrix}
\begin{pmatrix}
 \lambda_0 \\ \lambda_1 \\ \lambda_2 \\ \vdots \\ \vdots \\ \lambda_{d-1} 
\end{pmatrix}
 =
\begin{pmatrix}
 a_0 \\ a_1 \\ a_2 \\ \vdots \\ \vdots \\ a_{d-1} 
\end{pmatrix}.
\end{equation*}
By definition, $(y_n)$ satisfies the same recurrence as $(a_n)$, so $y_n=\sum_{j=1}^n c_j y_{n-j}$. Thus the inverse of the above $d\times d$ matrix is 
\begin{equation*}
\begin{pmatrix}
 1 & 0 &\cdots&\cdots& 0 \\
 -c_1 & 1 & 0 &\cdots& 0 \\
 \vdots & \vdots &\ddots&& \vdots \\
 \vdots & \vdots &&\ddots& \vdots \\
 -c_{d-1} & -c_{d-2} & \cdots & -c_1 & 1 
 \end{pmatrix},
\end{equation*}
and the claimed formula follows by applying this matrix to the vector $(a_0,\dots,a_{d-1})$.
\end{proof}

\begin{remark}
It is worth mentioning, that the representation \eqref{eq:BellRepresentation} provides a unifying approach to linear recurrence sequences in which the coefficients of the recurrence are separated and organized inside the partial Bell polynomials. In many cases, this gives known and new combinatorial identities for the sequence at hand (regardless of the order of recursion) as well as for their repeated convolutions, see Section~\ref{sec:Convolutions}.  
\end{remark}

For illustration purposes, let us consider a few basic examples.
\begin{example}(Generalized Fibonacci) \label{ex:genFibonacci}
For arbitrary $\alpha$ and coefficients $c_1$ and $c_2$, let $(f_n)$ be the sequence defined by
\begin{gather*}
  f_0=0,\;\; f_1=\alpha, \\
  f_n=c_1 f_{n-1}+c_2 f_{n-2} \;\text{ for } n\ge 2.
\end{gather*}
In the terminology of Proposition~\ref{prop:BellRepresenation}, we then have $\lambda_0=0$, $\lambda_1=\alpha$, and for $n\ge 1$ we get 
\begin{equation*}
f_n=\alpha y_{n-1} = \alpha \sum_{k=0}^{n-1} \frac{k!}{(n-1)!} B_{n-1,k}(1!c_1, 2!c_2, 0, \dots),
\end{equation*}
and since $B_{n-1,k}(c_1,2c_2,0,\dots)=\frac{(n-1)!}{k!} \binom{k}{n-1-k} c_1^{2k-n+1}c_2^{n-1-k}$, we arrive at
\begin{equation}
 f_n = \alpha \sum_{k=0}^{n-1} \binom{k}{n-1-k} c_1^{2k-n+1}c_2^{n-1-k}
 = \alpha \sum_{j=0}^{n-1} \binom{n-1-j}{j}c_1^{n-1-2j} c_2^j.
\end{equation}
\end{example}
\begin{example}(Padovan, A000931 in \cite{Sloane}) \label{ex:Padovan}
Consider the sequence defined by
\begin{gather*}
  P_0=1, \;\; P_1=P_2=0, \\
  P_n=P_{n-2}+P_{n-3} \;\text{ for } n\ge 3.
\end{gather*}
Using Proposition~\ref{prop:BellRepresenation} with $c_1=0$ and $c_2=c_3=1$, we get $\lambda_0=1$, $\lambda_1=0$, $\lambda_2=-1$, and for $n\ge 3$,
\[ P_n = y_{n}-y_{n-2}=y_{n-3}=\sum_{k=0}^{n-3} \frac{k!}{(n-3)!} B_{n-3,k}(0,2!,3!,0 \dots). \] 
Now, since $\frac{k!}{(n-3)!} B_{n-3,k}(0,2!,3!,0 \dots)=\binom{k}{n-3-2k}$, we conclude
\begin{equation}
P_n = \sum_{k=0}^{n-3} \binom{k}{n-3-2k} \;\text{ for } n\ge 3.
\end{equation}
\end{example}
\begin{example}(Tribonacci, A000073 in \cite{Sloane}) \label{ex:Tribonacci}
Let $(t_n)$ be the sequence defined by
\begin{gather*}
  t_0=t_1=0,\;\; t_2=1, \\
  t_n=t_{n-1}+t_{n-2}+t_{n-3} \;\text{ for } n\ge 3.
\end{gather*}
By Proposition~\ref{prop:BellRepresenation}, we have $t_n=y_{n-2}$, so
\begin{equation*}
 t_n= \sum_{k=0}^{n-2} \frac{k!}{(n-2)!} B_{n-2,k}(1!,2!,3!,0,\dots),
\end{equation*}
and since $B_{n,k}(1!,2!,3!,0,\dots)=\frac{n!}{k!}\sum_{\ell=0}^k \binom{k}{k-\ell} \binom{k-\ell}{n+\ell-2k}$, we get
\begin{equation*}
 t_n =\sum_{k=0}^{n-2} \sum_{\ell=0}^k \binom{k}{k-\ell} \binom{k-\ell}{n-2+\ell-2k}
 =\sum_{k=0}^{n-2} \sum_{\ell=0}^k \binom{k}{\ell} \binom{\ell}{n-2-k-\ell}.
\end{equation*}
With the change of variable $\ell=j-k$, and changing the order of summation, we arrive at
\begin{equation}
 t_n = \sum_{j=0}^{n-2} \sum_{k=0}^{j} \binom{k}{j-k} \binom{j-k}{n-2-j} \;\text{ for } n\ge 2.
\end{equation}
\end{example}

\begin{example}(Chebyshev) 
We now consider the Chebyshev polynomials defined by
\begin{gather*}
 T_0(x)=1, \;\; T_1(x)=x, \\
 T_n(x)=2xT_{n-1}(x)-T_{n-2}(x) \;\text{ for } n\ge 2.
\end{gather*}
Here $c_1=2x$, $c_2=-1$, so $\lambda_0=1$, $\lambda_1=-x$, and for $n\ge 1$ we get
\begin{align*}
T_n(x) &= y_{n}-xy_{n-1} \\
&= \sum_{j=0}^{n} \frac{j!}{n!} B_{n,j}(2x, -2, 0, \dotsc) -x\sum_{j=0}^{n-1} \frac{j!}{(n-1)!} B_{n-1,j}(2x, -2, 0, \dotsc) \\
&= \sum_{j=0}^{n} \binom{j}{n-j}(2x)^{2j-n}(-1)^{n-j}-x\sum_{j=0}^{n-1} \binom{j}{n-1-j}(2x)^{2j-n+1}(-1)^{n-1-j} \\
&= \sum_{j=0}^{n} \binom{j}{n-j}(2x)^{2j-n}(-1)^{n-j}-\frac12\sum_{j=1}^{n} \binom{j-1}{n-j}(2x)^{2j-n}(-1)^{n-j},
\end{align*}
which  can be written as
\begin{equation}\label{eq:Chebyshev1}
T_n(x)= \sum_{k=0}^{\lfloor\frac{n}{2}\rfloor} (-1)^{k}\frac{n}{2(n-k)}\binom{n-k}{k}(2x)^{n-2k} \;\text{ for } n\ge 1.
\end{equation}

Similarly, for the Chebyshev polynomials of second kind $U_n(x)$, defined by the same recurrence relation as for $T_n(x)$, but with initial values $U_0(x)=1$ and $U_1(x)=2x$, we get
\begin{equation}\label{eq:Chebyshev2}
U_n(x) = y_n = \sum_{k=0}^{\lfloor\frac{n}{2}\rfloor} (-1)^{k}\binom{n-k}{k}(2x)^{n-2k} \;\text{ for } n\ge 1.
\end{equation}
\end{example}

\subsection*{Power Sums}
We finish this section by considering the power sum symmetric functions 
\begin{equation*}
 s_n=x_1^n+\cdots+x_d^n
\end{equation*}
with $d$ variables. The sequence $(s_n)$ satisfies the relations (Newton's identities):
\begin{align*}
 s_0&=d \\
 s_1&=e_1 \\
 s_2&=e_1s_1-2e_2 \\
 s_3&=e_1s_2-e_2s_1+3e_3 \\
 &\;\; \vdots \\
 s_n&=e_1s_{n-1}-e_2 s_{n-2}+\cdots+(-1)^{d-1}e_d s_{n-d} \;\text{ for } n\ge d,
\end{align*}
where $e_1,\dots,e_d$ are the elementary symmetric functions in $x_1,\dots,x_d$. In other words, $(s_n)$ is a linear recurrence sequence of length $d$ with initial values $s_0,s_1,\dots,s_{d-1}$. Thus, as a consequence of Proposition~\ref{prop:BellRepresenation}, each power sum $s_n$ can be expressed in terms of partial Bell polynomials in $e_1,\dots,e_d$. This representation is an efficient way to organize and prove the Girard-Waring formulas (see e.g. \cite{Gould97}).

\begin{proposition}\label{prop:PowerSums}
Let $s_n=x_1^n+\cdots+x_d^n$ and let $e_1,\dots,e_d$ be the elementary symmetric functions in the variables $x_1,\dots,x_d$. Then, for $n\ge 1$, we have
\begin{align*}
 s_n &= \sum_{k=1}^n (-1)^{n+k} \frac{(k-1)!}{(n-1)!} B_{n,k}(1!e_1,2!e_2,\dots,d!e_d,0,\dots) .
\end{align*}
\end{proposition}
\medskip
Our proof relies on the following basic recursive formula.
\begin{lemma} \label{lem:keyIdentity}
For any sequence $x=(x_1,x_2,\dots)$, we have
  \begin{equation*}
   nB_{n, k}(x) = \sum_{j=1}^{n-k+1} j \binom{n}{j} x_j B_{n-j, k-1}(x).
  \end{equation*}
\end{lemma}
\begin{proof}
This is a consequence of the known identity
\begin{equation*}
  B_{n,k}=\sum_{j=0}^{n-k}\binom{n-1}{j}x_{j+1}B_{n-1-j,k-1}=\sum_{j=1}^{n-k+1}\binom{n-1}{j-1}x_{j}B_{n-j,k-1},
\end{equation*}
see equation (11.11) on p.~415 in \cite{Charalambides}. Therefore,
\begin{equation*}
 n B_{n,k} = \sum_{j=1}^{n-k+1} n\binom{n-1}{j-1} x_{j}B_{n-j,k-1} 
 = \sum_{j=1}^{n-k+1} j\binom{n}{j} x_{j}B_{n-j,k-1} 
\end{equation*}
as claimed.
\end{proof}
\begin{proof}[{\bf Proof of Proposition~\ref{prop:PowerSums}}]
Let $s_n= x_1^n+\cdots+x_d^n$. As mentioned before, $(s_n)$ is a linear recurrence sequence with coefficients \begin{equation*}
 c_j = (-1)^{j-1} e_j \; \text{ for } j=1,\dots, d,
\end{equation*}  
and initial values 
\begin{equation*}
 s_0=d, \quad s_k=\sum_{j=1}^{k-1} c_j s_{k-j} + kc_k \;\text{ for } k=1,\dots,d-1.
\end{equation*}
By Proposition~\ref{prop:BellRepresenation}, we can write $s_n=\sum_{j=0}^{d-1} \lambda_j y_{n-j}$ with $(y_n)$ defined as in \eqref{eq:canonicalSeq} and the $\lambda_j$'s given by $\lambda_0=s_0=d$ and
\begin{equation*}
 \lambda_j=s_j-\sum_{i=1}^j c_i s_{j-i} = \bigg(\sum_{i=1}^{j-1} c_i s_{j-i} + jc_j\bigg)-\sum_{i=1}^j c_i s_{j-i}
 = (j-d) c_j 
\end{equation*}
for $j=1,\dots,d-1$. Also, recall that $(y_n)$ is designed to satisfy the same recurrence relation as $(s_n)$ for $n\ge d$. Therefore, $y_n=\sum_{j=1}^{d} c_jy_{n-j}$ and we can rewrite $s_n$ as follows
\begin{align*}
 s_n&=\sum_{j=0}^{d-1} \lambda_j y_{n-j} \\
 &=d\bigg(\sum_{j=1}^{d} c_jy_{n-j}\bigg) + \sum_{j=1}^{d-1} (j-d)c_j y_{n-j} 
 \quad\quad \big[\lambda_j=(j-d)c_j\big]\\
 &=dc_d y_{n-d} +  \sum_{j=1}^{d-1} jc_j y_{n-j} = \sum_{j=1}^{d} jc_j y_{n-j} \\
 &= \sum_{j=1}^{d} \sum_{k=0}^{n-j}jc_j \frac{k!}{(n-j)!} B_{n-j,k}(1!c_1, 2!c_2, \dots, d!c_d,0,\dots).
\end{align*}
Since $c_j=0$ for $j>d$, we can write the sum over $j$ up to $n$ to facilitate a change of summation. Then 
\begin{align*}
 s_n&= \sum_{j=1}^{n} \sum_{k=0}^{n-j}jc_j \frac{k!}{(n-j)!} B_{n-j,k}(1!c_1, 2!c_2, \dots) \\
 &= \sum_{k=0}^{n-1} \sum_{j=1}^{n-k}jc_j \frac{k!}{(n-j)!} B_{n-j,k}(1!c_1, 2!c_2, \dots) \\
 &= \sum_{k=0}^{n-1} \frac{k!}{n!} \bigg[\sum_{j=1}^{n-k}j\binom{n}{j} (j!c_j) B_{n-j,k}(1!c_1, 2!c_2, \dots)\bigg]
\end{align*}
which by Lemma~\ref{lem:keyIdentity} yields
\begin{equation*}
 s_n =\sum_{k=0}^{n-1} \frac{k!}{n!}\, nB_{n,k+1}(1!c_1, 2!c_2, \dots) = 
 \sum_{k=1}^{n} \frac{(k-1)!}{(n-1)!} B_{n,k}(1!c_1, 2!c_2, \dots).
\end{equation*}
The claimed identity for the power sum $s_n$ follows by replacing back $c_j=(-1)^{j-1} e_j$ and using the homogeneity properties of the polynomial $B_{n,k}$.
\end{proof}

%%%%%%%%%%%%%%%%%%%%%%%%%%%%%%%%%%%%%%%%%%%%
\section{Convolutions}
\label{sec:Convolutions}

We now turn our attention to convolutions of sequences of the form  
\begin{equation}\label{eq:canonicalSeq2}
 y_n = \sum_{k=0}^n \frac{k!}{n!} B_{n,k}(1!c_1, 2!c_2, 3!c_3, \dots).
\end{equation}
In \cite{BGW14a}, we considered a more general family of sequences and proved the following result: 
\begin{theorem}[{\cite[Thm.~2.1]{BGW14a}}] 
Let $a$ and $b$ be arbitrary numbers. Let $y_0=1$ and 
\begin{equation*}
 y_n = \sum_{k=1}^n \binom{a n+b k}{k-1}\frac{(k-1)!}{n!} B_{n,k}(1!c_1, 2!c_2, \dots)
 \;\text{ for $n\ge 1$.}
\end{equation*}
For $r\in\mathbb{N}$, we have
\begin{equation*}
\sum_{m_1+\dots+m_r=n} \!\! y_{m_1}\cdots y_{m_r} 
 = r\sum_{k=1}^{n} \binom{a n+b k+r-1}{k-1} \frac{(k-1)!}{n!} B_{n,k}(1!c_1, 2!c_2, \dots).
\end{equation*}
\end{theorem}

\medskip
The proof of this theorem relies on a convolution formula for partial Bell polynomials given by the authors in \cite{BGW12}. In particular, the special case when $a=0$ and $b=1$ can be formulated as follows.

\begin{corollary}\label{cor:r_convolution}
For $(y_n)$ defined by \eqref{eq:canonicalSeq2} and $r\in\mathbb{N}$, we have
\begin{align}\label{eq:r_convolution}
y_n^{(r)}= \sum_{m_1+\dots+m_r=n} \!\! y_{m_1}\cdots y_{m_r} 
 &=r\sum_{k=1}^{n} \binom{k+r-1}{k-1} \frac{(k-1)!}{n!} B_{n,k}(1!c_1, 2!c_2, \dots) \\ \notag
 &=\sum_{k=1}^{n} \binom{k+r-1}{k} \frac{k!}{n!} B_{n,k}(1!c_1, 2!c_2, \dots).
\end{align}
\end{corollary}

\medskip
\begin{remark}
More generally, if $\delta\ge 0$ is an integer, and if we let $y_{-1}=\cdots=y_{-\delta}=0$, then
\begin{equation*}
\sum_{m_1+\dots+m_r=n} \!\! y_{m_1-\delta}\cdots y_{m_r-\delta}
  =\sum_{k=0}^{n-\delta r} \binom{k+r-1}{k}\frac{k!}{(n-\delta r)!} B_{n-\delta r,k}(1!c_1, 2!c_2, \dots).
\end{equation*}
\end{remark}

\bigskip
Let us now revisit some of the basic examples considered in the previous section.
\begin{example}(Generalized Fibonacci) 
For $\alpha,c_1,c_2\in\mathbb{R}$, consider $(f_n)$ defined by
\begin{equation*}
  f_0=0,\;\; f_1=\alpha, \quad f_n=c_1 f_{n-1}+c_2 f_{n-2} \;\text{ for } n\ge 2.
\end{equation*}
Then, as described in Example~\ref{ex:genFibonacci}, we have $f_n=\alpha y_{n-1}$, and therefore
\begin{align*}
 \sum_{m_1+\dots+m_r=n} \!\! f_{m_1}\cdots f_{m_r}
 &= \alpha^r \sum_{m_1+\dots+m_r=n} \!\! y_{m_1-1}\cdots y_{m_r-1} \\
 &= \alpha^r \sum_{k=0}^{n-r} \binom{k+r-1}{k} \frac{k!}{(n-r)!} B_{n-r,k}(1!c_1,2!c_2,0,\dots) \\
 &= \alpha^r \sum_{k=0}^{n-r} \binom{k+r-1}{k} \binom{k}{n-r-k} c_1^{2k-n+r}c_2^{n-r-k}.
\end{align*}
\end{example}
\begin{example}(Padovan) \label{ex:Padovan_convolution}
Let $(P_n)$ be defined by
\begin{gather*}
  P_0=1, \;\; P_1=P_2=0, \quad P_n=P_{n-2}+P_{n-3} \;\text{ for } n\ge 3.
\end{gather*}
As mentioned in Example~\ref{ex:Padovan}, we have $P_n = y_{n-3}$ and so
\begin{align*}
 \sum_{m_1+\dots+m_r=n} \!\! P_{m_1+1}\cdots P_{m_r+1}
 &= \sum_{m_1+\dots+m_r=n} \!\! y_{m_1-2}\cdots y_{m_r-2} \\
 &= \sum_{k=0}^{n-2r} \binom{k+r-1}{k}\frac{k!}{(n-2r)!} B_{n-2r,k}(0,2!,3!,0, \dots) \\
 &= \sum_{k=0}^{n-2r} \binom{k+r-1}{k} \binom{k}{n-2r-2k}.
\end{align*}
And, with a little more work, we also get
\begin{equation*}
 \sum_{m_1+\dots+m_r=n} \!\! P_{m_1}\cdots P_{m_r}
 = \sum_{\ell=1}^{r}\binom{r}{\ell}\sum_{k=0}^{n-3\ell} \binom{k+\ell-1}{k} \binom{k}{n-3\ell-2k}.
\end{equation*}
\end{example}
\begin{example}(Tribonacci) \label{ex:Tribonacci_convolution}
Let $(t_n)$ be defined by
\begin{gather*}
  t_0=t_1=0,\;\; t_2=1, \quad t_n=t_{n-1}+t_{n-2}+t_{n-3} \;\text{ for } n\ge 3.
\end{gather*}
As discussed in the previous section, we have $t_n=y_{n-2}$, so
\begin{align*}
\sum_{m_1+\dots+m_r=n} \!\! t_{m_1}\cdots t_{m_r}
 &= \sum_{m_1+\dots+m_r=n} \!\! y_{m_1-2}\cdots y_{m_r-2} \\
 &= \sum_{k=0}^{n-2r} \binom{k+r-1}{k}\frac{k!}{(n-2r)!} B_{n-2r,k}(1!,2!,3!,0, \dots) \\
 &= \sum_{k=0}^{n-2r}\sum_{\ell=0}^k \binom{k+r-1}{k} \binom{k}{\ell} \binom{\ell}{n-2r-k-\ell}.
\end{align*}
\end{example}

\bigskip
We now present a recurrence relation for convolution sequences of the form \eqref{eq:r_convolution}. 
\begin{theorem} \label{thm:linearRecurrence}
For any sequence of the form $y_n = \sum\limits_{k=0}^n \frac{k!}{n!} B_{n,k}(1!c_1, 2!c_2, \dots)$,
and for $r\in\mathbb{N}$, consider the convolution sequence
\begin{equation*}
y_n^{(r)}= \sum_{m_1+\dots+m_r=n} \!\! y_{m_1}\cdots y_{m_r} \;\text{ for } n\ge 0.
\end{equation*}
Then, for $n\ge1$, we have the recurrence relation
\begin{equation} \label{eq:linearRecurrence}
 n\, y_n^{(r)} = \sum_{m=1}^{n} [n+m(r-1)]c_m\, y_{n-m}^{(r)}.
\end{equation}
\end{theorem}
\begin{proof} 
By definition, we have $y_0^{(r)}=1$ and $y_1^{(r)}=c_1$, so \eqref{eq:linearRecurrence} is true for $n=1$. For $n>1$, write 
\[ \sum_{m=1}^{n} [n+m(r-1)]c_m\, y_{n-m}^{(r)}=nrc_n+\sum_{m=1}^{n-1} [n+m(r-1)]c_m\, y_{n-m}^{(r)}. \]
Now, by means of the identity \eqref{eq:r_convolution}, we have
\begin{align*}
\frac{1}{r}\sum_{m=1}^{n-1}[n &+ m (r-1)] c_m   y_{n-m}^{(r)}\\
&=\frac{1}{r}\sum_{m=1}^{n-1}[n+m (r-1)] c_m 
   \sum_{k=1}^{n-m}r\tbinom{k+r-1}{k-1} \tfrac{(k-1)!}{(n-m)!}  B_{n-m,k}(1!c_1, 2!c_2,\dots) \\
&=\sum_{k=1}^{n-1}\tbinom{k+r-1}{k-1}(k-1)!
   \sum_{m=1}^{n-k} \tfrac{[n+m (r-1)]}{(n-m)!} c_m B_{n-m,k}(1!c_1, 2!c_2,\dots)\\
&=\sum_{k=2}^{n}\tbinom{k+r-2}{k-2}(k-2)!
   \sum_{m=k-1}^{n-1} \tfrac{[n+(n-m) (r-1)]}{m!} c_{n-m}  B_{m,k-1}(1!c_1, 2!c_2,\dots).
\end{align*}
We split the last equation into two terms: 
\begin{equation*}
= \sum_{k=2}^{n}\tbinom{k+r-2}{k-2}(k-2)!\left[\sum_{m=k-1}^{n-1} \tfrac{n}{m!} c_{n-m}  B_{m,k-1}
  + \sum_{m=k-1}^{n-1} \tfrac{(n-m) (r-1)}{m!} c_{n-m}  B_{m,k-1}\right].
\end{equation*}
On the one hand,
\begin{align*}
\sum_{k=2}^{n}\tbinom{k+r-2}{k-2} 
& (k-2)! \sum_{m=k-1}^{n-1} \tfrac{n}{m!} c_{n-m}  B_{m,k-1}(1!c_1, 2!c_2,\dots) \\
&= \sum_{k=2}^{n}\tbinom{k+r-2}{k-2}\tfrac{(k-2)!\,k}{(n-1)!}
     \left[\frac1{k}\!\sum_{m=k-1}^{n-1}\binom{n}{m}(n-m)!\, c_{n-m}  B_{m,k-1}(1!c_1, 2!c_2,\dots)\right] \\
&= \sum_{k=2}^{n}\tbinom{k+r-2}{k-2}\tfrac{(k-2)!\,k}{(n-1)!} B_{n,k}(1!c_1, 2!c_2,\dots).
\end{align*}
On the other hand,
\begin{align*}
\sum_{k=2}^{n}\tbinom{k+r-2}{k-2}
& (k-2)!\sum_{m=k-1}^{n-1} \tfrac{(n-m) (r-1)}{m!} c_{n-m}  B_{m,k-1} \\
&= \sum_{k=2}^{n}\tbinom{k+r-2}{k-2}\tfrac{(k-2)!(r-1)}{(n-1)!}
     \left[\sum_{m=k-1}^{n-1}\tbinom{n-1}{m} (n-m)!  c_{n-m}  B_{m,k-1}(1!c_1, 2!c_2,\dots)\right] \\
&= \sum_{k=2}^{n}\tbinom{k+r-2}{k-2}\tfrac{(k-2)!(r-1)}{(n-1)!} B_{n,k}(1!c_1, 2!c_2,\dots).
\end{align*}
Therefore,
\begin{align*}
\frac{1}{r}\sum_{m=1}^{n-1}[n + m (r-1)]c_m y_{n-m}^{(r)}
 &= \sum_{k=2}^{n}\tbinom{k+r-2}{k-2}\tfrac{(k-2)!}{(n-1)!}(k+r-1) B_{n,k}(1!c_1, 2!c_2,\dots) \\
 &= \sum_{k=2}^{n}\tbinom{k+r-1}{k-1}\tfrac{(k-1)!}{(n-1)!}B_{n,k}(1!c_1, 2!c_2,\dots),
\end{align*}
and so
\begin{align*}
\sum_{m=1}^{n-1}[n + m (r-1)]c_m y_{n-m}^{(r)}
 &= r\sum_{k=2}^{n}\tbinom{k+r-1}{k-1}\tfrac{(k-1)!}{(n-1)!}B_{n,k}(1!c_1, 2!c_2,\dots) \\
 &= n\sum_{k=2}^{n}\tbinom{k+r-1}{k}\tfrac{k!}{n!}B_{n,k}(1!c_1, 2!c_2,\dots) = n(y_n^{(r)}-rc_n).
\end{align*}
Finally, adding the term $nrc_n$ to both sides of this equation, we arrive at \eqref{eq:linearRecurrence}.
\end{proof}

To illustrate our result, we now consider a few basic examples.

\begin{example}%(Fibonacci) 
Let $(a_n)$ be the sequence defined by
\begin{equation*}
  a_0=1,\;\; a_1=1,\quad a_n=a_{n-1}+a_{n-2} \;\text{ for } n\ge 2.
\end{equation*}
This is a shift of the Fibonacci sequence ($a_n=F_{n+1}$), and $a_n= \sum\limits_{k=0}^n \frac{k!}{n!} B_{n,k}(1, 2, 0, \dots)$. 

By means of Corollary~\ref{cor:r_convolution}, we have
\begin{equation*}
 a_n^{(r)}=\sum\limits_{m_1+\dots+m_r=n} \!\! a_{m_1}\cdots a_{m_r}
 = \sum_{k=1}^{n} \binom{k+r-1}{k} \binom{k}{n-k},
\end{equation*} 
and according to Theorem~\ref{thm:linearRecurrence}, this sequence satisfies
\begin{equation*}
  n a_n^{(r)} = (n+r-1)a_{n-1}^{(r)} + (n+2(r-1))a_{n-2}^{(r)}.
\end{equation*} 
For $r=2,3,4$, this gives the recurrence relations (cf. \cite{Sloane}):
\begin{align*}
  n a_n^{(2)} &= (n+1)a_{n-1}^{(2)} + (n+2)a_{n-2}^{(2)}, \quad \text{(A001629)}\\
  n a_n^{(3)} &= (n+2)a_{n-1}^{(3)} + (n+4)a_{n-2}^{(3)}, \quad \text{(A001628)}\\
  n a_n^{(4)} &= (n+3)a_{n-1}^{(4)} + (n+6)a_{n-2}^{(4)}. \quad \text{(A001872)}
\end{align*} 
\end{example}
\begin{example}%(Padovan)
Let $(a_n)$ be the sequence defined by
\begin{gather*}
  a_0=1,\; a_1=0,\; a_2=1, \text{ and} \\
  a_n=a_{n-2}+a_{n-3} \;\text{ for } n\ge 3.
\end{gather*}
This is a shifted version of the Padovan sequence, and $a_n= \sum\limits_{k=0}^n \frac{k!}{n!} B_{n,k}(0, 2!, 3!, 0, \dots)$. 

By Theorem~\ref{thm:linearRecurrence}, the corresponding convolution sequence (cf. Example~\ref{ex:Padovan_convolution})
\begin{equation*}
  a_n^{(r)}= \sum_{k=1}^{n} \binom{k+r-1}{k} \binom{k}{n-2k}
\end{equation*} 
satisfies the recurrence relation
\begin{equation*}
  n a_n^{(r)} = (n+2(r-1))a_{n-2}^{(r)} + (n+3(r-1))a_{n-3}^{(r)}.
\end{equation*}
For $r=2$, we get a shift of the sequence A228577 in \cite{Sloane} (number of gaps of length $1$ in all possible covers of a line of length $n$ by segments of length $2$). In this case, we obtain
\begin{equation*}
  n a_n^{(2)} = (n+2)a_{n-2}^{(2)} + (n+3)a_{n-3}^{(2)}.
\end{equation*}
\end{example}

\begin{example}%(Tribonacci)
Let $(a_n)$ be the sequence defined by
\begin{gather*}
  a_0=1,\; a_1=1,\; a_2=2, \text{ and} \\
  a_n=a_{n-1}+a_{n-2}+a_{n-3} \;\text{ for } n\ge 3.
\end{gather*}
This is a shift of the Tribonacci sequence discussed in Examples~\ref{ex:Tribonacci} and \ref{ex:Tribonacci_convolution}. More precisely, $a_n=t_{n+2}=\sum\limits_{k=0}^n \frac{k!}{n!} B_{n,k}(1!, 2!, 3!, 0, \dots)$, and the convolution sequence $(a_n^{(r)})$ takes the form 
\begin{equation*}
  a_n^{(r)} = \sum_{k=1}^{n} \binom{k+r-1}{k} \sum_{\ell=0}^k \binom{k}{k-\ell} \binom{k-\ell}{n+\ell-2k},
\end{equation*} 
which satisfies
\begin{equation*} 
 n a_n^{(r)} = (n+r-1) a_{n-1}^{(r)}+(n+2(r-1)) a_{n-2}^{(r)}+(n+3(r-1)) a_{n-3}^{(r)}.
\end{equation*}
For $r=2$, we obtain a recurrence relation for A073778 in \cite{Sloane}:
\begin{equation*} 
 n a_n^{(2)} = (n+1) a_{n-1}^{(2)}+(n+2) a_{n-2}^{(2)}+(n+3) a_{n-3}^{(2)}.
\end{equation*}
\end{example}

%%%%%%%%%%%%%%%%%%%%%%%%%%%%%%%%%%%%%%%%%%%%


\begin{thebibliography}{99}
\bibitem{Bell} E.T.~Bell, \emph{Exponential polynomials}, Ann. of Math. \textbf{35} (1934), pp. 258--277.
%
\bibitem{BernsteinSloane} M.~Bernstein and N.J.A.~Sloane, {\em Some canonical sequences of integers}, Linear Algebra Appl. \textbf{226/228} (1995), 57--72. 
%
\bibitem{BGW12} D.~Birmajer, J.~Gil, and M.~Weiner, \emph{Some convolution identities and an inverse relation involving partial Bell polynomials}, Electron. J. Combin. \textbf{19} (2012), no. 4, Paper 34, 14 pp.
%
\bibitem{BGW14a} D.~Birmajer, J.~Gil, and M.~Weiner, \emph{Convolutions of Tribonacci, Fuss--Catalan, and Motzkin sequences}, 2014, preprint, submitted to the Fibonacci Quarterly.
%
\bibitem{Cameron} P.J.~Cameron, \emph{Some sequences of integers}, Discrete Math. \textbf{75} (1989), no. 1-3, 89--102. 
% 
\bibitem{Charalambides} C.A.~Charalambides, \emph{Enumerative Combinatorics}, Chapman and Hall/CRC, Boca Raton, 2002.
%
\bibitem{Comtet} L.~Comtet, \emph{Advanced Combinatorics: The Art of Finite and Infinite Expansions}, D. Reidel Publishing Co., Dordrecht, 1974.
%
\bibitem{Gould97} H.~Gould, \emph{The Girard-Waring power sum formulas for symmetric functions and Fibonacci sequences}, Fibonacci Quart. \textbf{37} (1999), no. 2, 135--140. 
%
\bibitem{Sloane} N.J.A.~Sloane, The On-Line Encyclopedia of Integer Sequences, http://oeis.org/.
\end{thebibliography}
\end{document}